\documentclass[preprint,review,10pt]{elsarticle}

\usepackage{amsmath,amssymb,amsfonts,amsthm}
\usepackage{hyperref}
\usepackage{graphicx}
\usepackage{mathrsfs}

\usepackage{float}
\usepackage{color}
\usepackage{cases}
\usepackage{multirow}

\newtheorem{lemma}{Lemma}
\newtheorem{remark}{Remark}
\newtheorem{theorem}{Theorem}

\begin{document}

\begin{frontmatter}



\title{Uniform convergence of optimal order under a balanced norm of a local discontinuous Galerkin method on a Shishkin mesh}

\tnotetext[funding]{
The current research  was partly supported by NSFC (11771257), Shandong Provincial NSF (ZR2021MA004).
}

\author[label1] {Jin Zhang \corref{cor1}}
\author[label1] {Wenchao Zheng \fnref{cor2}}
\cortext[cor1] {Corresponding author:  jinzhangalex@hotmail.com }
\fntext[cor2] {Email: Superwenchao@hotmail.com}

\address[label1]{School of Mathematical Sciences, Shandong Normal University, Jinan 250014, Shandong Province, PR China}

\begin{abstract}
For singularly perturbed reaction-diffusion problems in 1D and 2D, we study a local discontinuous Galerkin (LDG) method on a Shishkin mesh. In these cases, the standard energy norm is too weak to capture adequately the behavior of the boundary layers that appear in the solutions. To deal with this deficiency, we introduce a  balanced norm stronger than the energy norm. In order to achieve optimal convergence under the balanced norm in one-dimensional case, we design novel numerical fluxes and propose a special interpolation that consists of a Gauss-Radau projection and a local $L^2$ projection. Moreover, we generalize the numerical fluxes and interpolation, and extend convergence analysis of  optimal order from 1D to 2D. Finally, numerical experiments are presented to confirm the theoretical results.
\end{abstract}


\end{frontmatter}


%
%
%

\section{Introduction}
Singularly perturbed problems appear widely in many fields, such as fluid mechanics, electronic science, and energy development \cite{B2007-Singularly,V2005-Methods}. The typical characteristic of their solutions  is the presence of layers. 
To fully resolve layers and obtain uniform convergence with respect to singular perturbation parameters, layer-adapted meshes have been introduced since 1960s \cite{B1969-Towards} and have been an active research field \cite{Lin2010-Layer,R2019-Layer}.
Among them, the Shishkin meshes  have been widely used and analyzed \cite{miller1996-fitted,l2001-numerical,r1998-grids},  because it has a very simple structure.


On Shishkin meshes, standard numerical methods such as finite element method and finite difference method have been well analyzed \cite{R2008-Robust}. These methods cannot compute the gradient of the solution accurately.  However, in many physical applications, the gradient 
 is also very important because it is related to some important physical quantities, such as surface friction and temperature and so on. To meet this requirement,  the local discontinuous Galerkin (LDG) method \cite{C1998-local} is a good candidate.
The LDG method,  as a kind of discontinuous Galerkin (DG) method, inherits many advantages of the DG method \cite{C2012-Discontinuous}, such as good stability, high-order accuracy and flexibility on $hp$ adaptivity. These advantages make the LDG method more attractive to deal with those solutions with  layers. 
In recent years, the LDG method has been used to singularly perturbed problems; see \cite{2009-numerical,Z2013-Convergence,z2013-point,z2011-uniformly}. 

In this manuscript, we focus on singularly perturbed reaction-diffusion equations. For these problems, Mei et al. \cite{m2021-convergence} and Cheng  \cite{c2022-balanced} have analyzed uniform convergence   under an energy norm and a balanced norm, respectively. It should be mentioned that the balanced norm is more appropriate for layers than the energy norm. 
 In \cite{c2022-balanced},  the authors  proved a uniform convergence of optimal order if  the smooth part of the solution belongs to the finite element space. However, 
 this condition is too strong to hold in general.


We aim to design a LDG method, which does not have the above drawback.
For this aim,  we have carefully designed a new numerical flux, which is the key to our LDG method.  Besides, we propose a special interpolation that consists of a Gauss-Radau projection and a local $L^2$ projection. On the basis of the numerical fluxes and the interpolation, we prove that the LDG solution is convergent under the balanced norm with optimal order.   Furthermore, we extend the LDG method and its convergence analysis from 1D to 2D.

The content of this paper is divided into the following parts. In Section 2, we prove the uniform convergence of the LDG method for a one-dimensional singularly perturbed problem. Firstly,  we design a specific numerical flux and introduce a LDG method.  Then the Shishkin mesh is defined and interpolations are presented. Finally, we get the optimal convergence order under the balanced norm. In Section 3, we prove the uniform convergence of the LDG method for a two-dimensional singularly perturbed problem, and the content distribution is similar to that in Section 2. In Section 4, we present numerical experiments to support the theoretical results. Throughout this article, $C$ denotes a general positive constant that is independent of $\varepsilon$ and $N$, where $\varepsilon$ is a perturbation parameter and $N$ is the number of meshes.

%
%

\section{LDG method for one-dimensional case}
We consider the following one-dimensional singularly perturbed problem:
\begin{equation}\label{eq1.1}
-\varepsilon u''+b(x)u=f(x),\quad \text{$x\in \Omega=(0,1)$},\quad u(0)=u(1)=0,
\end{equation}
where $0< \varepsilon\ll 1$ is a perturbation parameter, $b(x)\geqslant \beta^2 \ge 0$ and $f(x)$ are sufficiently smooth functions on $\bar{\Omega}$. And the solution of problem 
\eqref{eq1.1} is typically characterized by boundary layers of width $O(\sqrt{\varepsilon}\ln(1/\varepsilon))$ at $x=0$ and $x=1$.

\subsection{LDG method}
We divide the domain $\Omega=(0,1)$ into $N$ cells: $$0=x_0<x_1<\ldots<x_{N-1}<x_N=1.$$
Denote an arbitrary subinterval $[x_{j-1}, x_{j}]$ by $I_j$ for $j=1,\ldots,N$, and denote the step size of the subinterval as $h_j = x_{j} -x_{j-1}$. The discontinuous finite element space $V_N$ is defined as $$V_{N}=\{v\in L^2(\Omega):\text{$v|_{I_j}\in \mathbb{P}_k(I_j), 1\le j\le N$ } \},$$ where $\mathbb{P}_k(I_j)$ is the space of polynomials of degree at most $k\ge 1$ on $I_j$.
For $v\in V_N$, we define $v^{\pm}_j=\lim_{x\to x^{\pm}_j}v(x)$. Define a jump as: $[[v]]_j=v^{-}_j-v^{+}_j$ for $j=1,2,\ldots,N-1$, $[[v]]_0=-v^{+}_0$ and $[[v]]_N=v^{-}_N$.

Next, we will introduce the LDG method   for problem \eqref{eq1.1}. First, one has
\begin{equation*}
\begin{aligned}
\varepsilon^{-1}q=u',\quad \text{in $\Omega$}, \\
-q'+b(x)u=f(x),\quad \text{in $\Omega$}.
\end{aligned}
\end{equation*}
Define $\chi=(r,v)\in V_N \times V_N$ to be any test function. Let $\left< \cdot,\cdot\right>_{I_j}$ be the inner product in $L^2(I_j)$.
Find the numerical solution $W=(Q, U)\in V_N \times V_N$ such that the following variational forms hold in each element $I_j$,
\begin{equation}
\begin{aligned}\label{method:1}
&\varepsilon^{-1}\left<Q,r\right>_{I_j}+\left<U,r'\right>_{I_j}-\hat{U}_j r^{-}_j+\hat{U}_{j-1}r^{+}_{j-1}=0,\\
&\left<Q,v'\right>_{I_j}+\left<bU,v\right>_{I_j}-\hat{Q}_jv^{-}_j+\hat{Q}_{j-1}v^{+}_{j-1}=\left<f,v\right>_{I_j},
\end{aligned}
\end{equation}
 where the numerical fluxes $\hat{Q}$ and $\hat{U}$ are defined by
\begin{equation*}
\begin{aligned}
&\hat{Q}_j=
\left\{
\begin{aligned}
& Q^{+}_0+\lambda_0U^{+}_0, \quad &&\text{$j=0$},\\
& Q^{+}_j, \quad &&\text{$j=1,2,\ldots,N-1$},\\
& Q^{-}_N-\lambda_NU^{-}_N, \quad &&\text{$j=N$},
\end{aligned}
\right. \label{eq:Bakhvalov mesh-Roos}\\
&\hat{U}_j=
\left\{
\begin{aligned}
& 0, \quad &&\text{$j=0,N$},\\
&U_{\frac{3}{4}N}^{-}+\lambda_q[[Q]]_{\frac{3}{4}N},\quad &&\text{$j=\frac{3}{4}N$},\\
& U^{-}_j, \quad &&\text{$j=1,2,\ldots,\frac{3}{4}N-1,\frac{3}{4}N+1,\ldots,N-1$},\\
\end{aligned}
\right.\\
\end{aligned}
\end{equation*}
with the stabilization parameters $\lambda_0=\lambda_N=\sqrt{\varepsilon}$ and $\lambda_q=1/\sqrt{\varepsilon}$.
\begin{remark}
The definition of the numerical flux $\hat{U}_j$ is the key to convergence analysis under the balanced norm. Compared with \cite{c2022-balanced}, there is an additional term involved with the jump of $Q_{\frac{3}{4}N}$ appearing in the definition of $\hat{U}_{\frac{3}{4}N}$. With this new numerical flux, we can get optimal convergence order without the impossible condition in \cite{c2022-balanced}.
\end{remark}

We 
denote $\left<w,v\right>=\sum_{j=1}^N\left<w,v\right>_{I_j}$ and 
express the scheme \eqref{method:1} by a compact form: Find  $W=(Q,U)\in V_N\times V_N$ such that
\begin{equation*}
B(W;\chi)=\left<f,v\right>,\quad \forall \chi=(r,v)\in V_N\times V_N,
\end{equation*}
where
\begin{equation}\label{method:2}
\begin{aligned}
B(W;\chi)=&\left<b U,v\right>+\varepsilon^{-1}\left<Q,r\right>+\left<U,r'\right>+\sum_{j=1}^{N-1}U^{-}_j[[r]]_j+\left<Q,v'\right>+\sum_{j=0}^{N-1}Q^{+}_j[[v]]_j\\
&-(Q v)_N^{-}+\sum_{j\in \{0,N\}}\lambda_j [[u]]_j[[v]]_j+\lambda_q[[Q]]_{3N/4}[[r]]_{3N/4}.
\end{aligned}
\end{equation}

\subsection{Regularity of the solution}
\begin{lemma}\label{lem1}
(Referene \cite{miller1996-fitted}) For $j=0,1,\dots,k+2$, assume that the solution $u$ of \eqref{eq1.1} can be decomposed as $u(x)=\bar{u}(x)+u_{\varepsilon,1}(x)+u_{\varepsilon,2}(x)$ with $x\in \bar{\Omega}$, then
\begin{equation}\label{solution:1}
|\frac{d^j \bar{u}(x)}{d x^j}|\le C,\quad |\frac{d^j u_{\varepsilon,1}(x)}{d x^j}|\le C\varepsilon^{-\frac{j}{2}}e^{-\frac{\beta x}{\sqrt{\varepsilon}}},\quad |\frac{d^j u_{\varepsilon,2}(x)}{d x^j}|\le C\varepsilon^{-\frac{j}{2}}e^{-\frac{\beta (1-x)}{\sqrt{\varepsilon}}}.
\end{equation}
 Consequently, for $j=0,1,\dots,k+1$, $q=\bar{q}+q_{\varepsilon,1}+q_{\varepsilon,2}=\varepsilon \frac{d\bar{u}}{dx}+\varepsilon \frac{d u_{\varepsilon,1}}{dx}+\varepsilon \frac{d u_{\varepsilon,2}}{dx}$ and
\begin{equation}\label{solution:2}
|\frac{d^j \bar{q}(x)}{d x^j}|\le C\varepsilon,\quad |\frac{d^j q_{\varepsilon,1}(x)}{d x^j}|\le C\varepsilon^{-\frac{j-1}{2}}e^{-\frac{\beta x}{\sqrt{\varepsilon}}},\quad |\frac{d^j q_{\varepsilon,2}(x)}{d x^j}|\le C\varepsilon^{-\frac{j-1}{2}}e^{-\frac{\beta (1-x)}{\sqrt{\varepsilon}}}.
\end{equation}
\end{lemma}

\subsection{Shishkin mesh}
In this section, we will introduce a Shishkin mesh. 
Let $N\ge4$ be a multiple of 4
and the transition parameter be 
\begin{align*}
\tau=\frac{\sigma \sqrt{\varepsilon}\ln N}{\beta}\le \frac{1}{4},
\end{align*}
where $\sigma \ge k+1$ is a user-chosen parameter. 
Divide  $[\tau, 1-\tau]$ into $N/2$  equidistant subintervals. Furthermore, there are N/4 subintervals evenly distributed in both the intervals $[0, \tau ]$ and $[1-\tau, 1]$.
Then, the mesh points are given by
\begin{equation}\label{mesh:1}
x_j=
\left\{
\begin{split}
& 4\frac{\sigma \sqrt{\varepsilon} }{\beta} t_j\ln N, \quad &&\text{for $j=0,1,\ldots,\frac{N}{4}$},\\
&\tau+2(1-2\tau)(t_j-\frac{1}{4}),  &&\text{for $j=\frac{N}{4}+1,\ldots,\frac{3N}{4}$},\\
&1-4\frac{\sigma \sqrt{\varepsilon} }{\beta} (1-t_j)\ln N, \quad &&\text{for $j=\frac{3N}{4}+1,\ldots,N$},
\end{split}
\right.
\end{equation}
where $t_j\equiv \frac{j}{N}$ for $j=0,1,\ldots,N$.
On the  Shishkin mesh \eqref{mesh:1}, for $1\le j\le N$, we have $$C\sqrt{\varepsilon}N^{-1}\le h_j \le CN^{-1}.$$

%
%


%
%

\subsection{Projections}
In this section, we combine the local $L^2$ projection and Gauss-Radau projection to obtain a new interpolation for later analytical proof.
\begin{itemize}
\item
Local $L^2$ projection $\pi$. For any $z\in L^2({\Omega})$, $\pi z\in V_N$ is defined as:
\begin{equation*}
\begin{aligned}
\left<\pi z,v\right>_{I_j}=\left<z,v\right>_{I_j}, \quad \forall v\in \mathbb{P}_k(I_j),
\end{aligned}
\end{equation*}
for each element $I_j=(x_{j-1},x_j)$, $j=1,2,\ldots,N$.
\item
Gauss-Radau projection $\pi^{\pm}$. For any $z\in H^1(I_j)$, $\pi^{\pm} z\in \mathbb{P}_k(I_j)$ is defined as:
\begin{equation*}
\begin{aligned}
&\left<\pi^{\pm} z,v\right>_{I_j}=\left<z,v\right>_{I_j},\quad \forall v\in \mathbb{P}_{k-1}(I_j),\\
&(\pi^{+}z)^{+}_{j-1}=z^{+}_{j-1},\quad (\pi^{-}z)^{-}_j=z^{-}_j,
\end{aligned}
\end{equation*}
for each element $I_j$, $j=1,2,\ldots,N$.
\end{itemize}
The proof of the existence and uniqueness of Gauss-Radau projection is referred to reference \cite{A1999-An}.

For the analysis of the article, we define the following interpolations:
\begin{equation*}
\begin{aligned}
&P^{-}_N u|_{I_j}=
\left\{
\begin{aligned}
& \pi^{-}u, \quad &&\text{$j=1,2,\ldots,N/4,3N/4+1,\ldots,N-1$},\\
& \pi u,\quad &&\text{$j=N/4+1,\ldots ,3N/4,N$},\\
\end{aligned}
\right. \\
&P^{+}_N q|_{I_j}=
\left\{
\begin{aligned}
& \pi^{+}q, \quad &&\text{$j=2,3,\ldots,N$},\\
& \pi q,\quad &&\text{$j=1$}.\\
\end{aligned}
\right. 
\end{aligned}
\end{equation*}

Let $||v||^2=\sum_{j=1}^N||v||_{I_j}^2$ and $||v||_{I_j}^2=\left<v,v\right>_{I_j}$.
\begin{lemma}\label{lem2}
(Reference \citep[Lemma 2.3]{c2022-balanced}) On the Shishkin mesh \eqref{mesh:1} with $\sigma \ge k+1$, one has
\begin{align}
&||u-P^{-}_N u||\le C\left[\varepsilon^{1/4}(N^{-1}\ln N)^{k+1}+N^{-(k+1)}\right],\label{mesh:2}\\
&||q-P^{+}_N q||\le C\varepsilon^{3/4}(N^{-1}\ln N)^{k+1},\label{mesh:3}\\
&||u-P^{-}_N u||_{L^{\infty}(\Omega)}\le CN^{-(k+1)},\label{mesh:4}\\
&||q-P^{+}_N q||_{L^{\infty}(\Omega)}\le C\varepsilon^{1/2}N^{-(k+1)}.\label{mesh:5}
\end{align}
\end{lemma}

%
%
%

\subsection{Convergence analysis}
Recall ${w}=(q,u)$ to be the exact solution of the problem \eqref{eq1.1}.
From \eqref{method:2}, we define an energy norm by
\begin{equation}\label{norm:1}
|||w|||^2_{E}=B(w;w)= \varepsilon^{-1}||q||^2+||b(x)^{\frac{1}{2}}u||^2+\sum_{j\in\{0,N\}}\lambda_j[[u]]_j^2+\lambda_q[[q]]^2_{\frac{3N}{4}}.
\end{equation} 
We introduce a more powerful balanced norm based on this energy norm, which is defined by
\begin{equation}\label{norm:2}
|||w|||^2_{B}= \varepsilon^{-\frac{3}{2}}||q||^2+||b(x)^{\frac{1}{2}}u||^2+\sum_{j\in\{0,N\}}[[u]]_j^2+[[q]]^2_{\frac{3N}{4}}.
\end{equation}

\begin{theorem}\label{theorem1}
Assume $\varepsilon^{} \le CN^{-1}$. On the Shishkin mesh \eqref{mesh:1}, we take $\sigma \ge k+1$. Let ${W}=(Q, U)\in V_N \times V_N$ be the numerical solution of the LDG scheme \eqref{method:1}. Let ${w}=(q,u)$ be the exact solution of the problem \eqref{eq1.1} satisfying Lemma \ref{lem1}. Define ${e=w-W}$. There is a constant $C>0$ such that
\begin{align*}
&|||e|||_{E}\le  C\left(\varepsilon^{1/4}(N^{-1}\ln N)^{k+1}+\varepsilon^{1/2}N^{-k}+N^{-(k+1)}\right),\\
&|||e|||_{B}\le C\left(\varepsilon^{1/2}(N^{-1}\ln N)^{k+1}+\varepsilon^{1/4}N^{-k}+(N^{-1}\ln N)^{k+1}\right).
\end{align*}
\end{theorem}

\begin{proof}
We use interpolation $P^{+}_N q$ and $P^{-}_N u$ to divide the error $e$ into two parts: ${e}=(q-Q,u-U)=\bold{\eta-\xi}$ with 
\begin{equation*}
\begin{aligned}
&\bold{\eta}=(\eta_q,\eta_u)=(q-P^{+}_N q,u-P^{-}_N u),\\
&\bold{\xi}=(\xi_q,\xi_u)=(Q-P^{+}_N q,U-P^{-}_N u)\in V_N \times V_N.
\end{aligned}
\end{equation*}

Owing to the consistency of numerical fluxes and the regularity assumptions of $u$ and $q$, one has the error equation
\begin{equation}\label{norm:3}
B(\xi;\chi)=B(\eta;\chi),\quad \forall \chi=(r,v)\in V_N \times V_N.
\end{equation}
Set $\chi=\xi$ in \eqref{norm:3}. According to the integration by parts, we derive
\begin{equation}\label{norm:4}
B(\xi;\xi)=|||\xi|||_E^2.
\end{equation}
Furthermore, by \eqref{method:2}, we can get
\begin{equation*}
B(\eta;\xi)\equiv \sum_{i=1}^8 S_i.
\end{equation*}
Here,
\begin{equation*}
\begin{aligned}
&S_1=\varepsilon^{-1}\left<\eta_q,\xi_q\right>,\quad  &&S_2=\left<\eta_u,\xi'_q\right>,\\
&S_3=\left<\eta_q,\xi'_u\right>, \quad &&S_4=\left<b\eta_u,\xi_u\right>,\\
&S_5=\sum_{j=0}^{N-1}(\eta_q)^{+}_j[[\xi_u]]_j-(\eta_q)^{-}_N(\xi_u)^{-}_N,\quad &&S_6=\sum_{j=1}^{N-1}(\eta_u)^{-}_j[[\xi_q]]_j,\\
&S_7=\lambda_q[[\eta_q]]_{3N/4}[[\xi_q]]_{3N/4},\quad &&S_8=\sum_{j\in\{0,N\}}\lambda_j[[\eta_u]]_j[[\xi_u]]_j.
\end{aligned}
\end{equation*}
Then, we will analyze them separately.

Using the Cauchy–Schwarz inequality and \eqref{mesh:3}, one has
\begin{align*}
S_1\le C(\varepsilon^{-1/2}||\eta_q||)(\varepsilon^{-1/2}||\xi_q||)\le C\varepsilon^{1/4}(N^{-1}\ln N)^{k+1}|||\xi|||_E.
\end{align*}
Owing to the orthogonality of the projection, we have $S_2=S_3=0$.\\
By the Cauchy–Schwarz inequality and \eqref{mesh:4}, we have
\begin{align*}
S_4&=\left<b\eta_u,\xi_u\right>_{(0,x_{N/4})\cup (x_{3N/4+1},x_{N-1})}\\
   &\le C||\eta_u||_{L^{\infty}(\Omega)}(\sqrt{\varepsilon}\ln N)^{1/2}|||\xi|||_E\\
   &\le C\varepsilon^{1/4}(\ln N)^{1/2}N^{-(k+1)}|||\xi|||_E.
\end{align*}
From the Cauchy–Schwarz inequality, \eqref{mesh:5} and $\lambda_0=\lambda_N =\sqrt{\varepsilon}$, we can get
\begin{align*}
S_5&=(\eta_q)^{+}_0(\xi_u)^{+}_0-(\eta_q)^{-}_N(\xi_u)^{-}_N\\
   &\le C\left[\sum_{j\in\{0,N\}}\lambda_j^{-1}[[\eta_q]]^{2}_j \right]^{1/2}\left[\sum_{j\in\{0,N\}}\lambda_j [[\xi_u]]^{2}_j \right]^{1/2}\\
   &\le C \varepsilon^{-1/4}||\eta_q||_{L^{\infty}(\Omega)}|||\xi|||_E\\
   &\le C\varepsilon^{1/4}N^{-(k+1)}|||\xi|||_E.
\end{align*}
Using the Cauchy–Schwarz inequality, the trace inequality, \eqref{mesh:4}, $\varepsilon \le CN^{-1}$ and $\lambda_q=\varepsilon^{-1/2}$, we obtain
\begin{align*}
S_6&=\sum_{j=N/4+1}^{3N/4-1}(\eta_u)^{-}_j[[\xi_q]]_j+(\eta_u)^{-}_{3N/4}[[\xi_q]]_{3N/4}\\
     &\le C\sum_{j=N/4+1}^{3N/4-1}||\eta_u||_{L^{\infty}(I_j)}h_j^{-1/2}||\xi_q||_{I_j \cup I_{j+1}}+C\lambda_q^{-1/2}||\eta_u||_{L^{\infty}(3N/4)}\left[\lambda_q[[\xi_q]]^2_{3N/4}\right]^{1/2}\\
     &\le C\varepsilon^{1/2}N^{-k}|||\xi|||_E+C\varepsilon^{1/4}N^{-(k+1)}|||\xi|||_E.
\end{align*}
By the Cauchy–Schwarz inequality, \eqref{mesh:5} and $\lambda_q =\varepsilon^{-1/2}$, one has
\begin{align*}
S_7\le C \left[\lambda_q[[\eta_q]]^2_{3N/4}\right]^{1/2} \left[\lambda_q[[\xi_q]]^2_{3N/4}\right]^{1/2}\le C\varepsilon^{1/4}N^{-(k+1)}|||\xi|||_E.
\end{align*}
From the Cauchy–Schwarz inequality, \eqref{mesh:4} and $\lambda_0 =\lambda_N =\varepsilon^{1/2}$,  we have
\begin{equation*}
\begin{aligned}
S_8&=\sum_{j\in \{0,N\}}\lambda_j[[\eta_u]]_j[[\xi_u]]_j\\
      &\le C\left[\sum_{j\in \{0,N\}}\lambda_j[[\eta_u]]^{2}_j\right]^{\frac{1}{2}}\left[\sum_{j\in \{0,N\}}\lambda_j[[\xi_u]]^{2}_j\right]^{\frac{1}{2}}\\
      &\le C \varepsilon^{1/4}||\eta_u||_{L^{\infty}(\Omega)}|||\xi|||_E\\
      &\le C\varepsilon^{1/4}N^{-(k+1)}|||\xi|||_E.
\end{aligned}
\end{equation*}
We gather all the above estimates and get
\begin{equation*}
\begin{aligned}
B(\eta;\xi)\le C\varepsilon^{1/4}(N^{-1}\ln N)^{k+1}|||\xi|||_{E}.
\end{aligned}
\end{equation*}
Considering \eqref{norm:3} and \eqref{norm:4}, we have
\begin{equation}\label{norm:5}
|||\xi|||^2_{E}= B(\xi;\xi)=B(\eta;\xi)\le C\left(\varepsilon^{1/4}(N^{-1}\ln N)^{k+1}+\varepsilon^{1/2}N^{-k}\right)|||\xi|||_{E}.
\end{equation}
Therefore,  it is straightforward to derive
\begin{align}
&|||\xi|||_{E}\le C\varepsilon^{1/4}(N^{-1}\ln N)^{k+1}+C\varepsilon^{1/2}N^{-k},\label{norm:6}\\
&|||\xi|||_{B}\le C\varepsilon^{-1/4}|||\xi|||_E\le C(N^{-1}\ln N)^{k+1}+C\varepsilon^{1/4}N^{-k}. \label{norm:7}
\end{align} 
From Lemma \ref{lem2} and the definition of norms in \ref{norm:1} and \ref {norm:2}, we can get
\begin{align}
&|||\eta|||_{E}\le C\varepsilon^{1/4}(N^{-1}\ln N)^{k+1}+CN^{-(k+1)},\label{norm:8}\\
&|||\eta|||_{B}\le C\varepsilon^{1/2}(N^{-1}\ln N)^{k+1}+C(N^{-1}\ln N)^{k+1}.\label{norm:9}
\end{align} 
Combining the results of \eqref{norm:6}, \eqref{norm:7}, \eqref{norm:8} and \eqref{norm:9}, we prove Theorem \ref{theorem1}.
%
\end{proof}

\begin{remark}
In Theorem \ref{theorem1}, if  $\varepsilon^{} \le CN^{-4}$, then we have 
$$|||e|||_{E}\le C(N^{-1}\ln N)^{k+1},\quad
|||e|||_{B}\le C(N^{-1}\ln N)^{k+1}.$$
We obtain optimal convergence order without the impossible condition in \cite{c2022-balanced}.
\end{remark}

\section{LDG method for two-dimensional case}
We consider a two-dimensional singularly perturbed problem:
\begin{align}\label{eq3.1}
&-\varepsilon \Delta u+b(x,y)u=f(x,y),\quad \text{in $ \Omega =(0,1)^2$},\\
&u(0)=0,\quad \text{on $ \partial\Omega$}. \nonumber
\end{align}
where $0 < \varepsilon \ll 1$ is a perturbation parameter,  $b(x, y) \geqslant 2\beta ^2 > 0$ and $f(x,y)$ are sufficient smooth functions on $\bar{\Omega}$. And it is well-known that the problem has a unique solution, which is characterized by the presence of  boundary layers of width  $O(\sqrt{\varepsilon}\ln(1/\varepsilon))$ along the entire boundary $ \partial\Omega$. 

\subsection{LDG method}
The two-dimensional layer-adapted meshes $\{(x_i
, y_j ), i,j=1,\ldots,N\}$ is constrcuted by the tensor-product of the one-dimensional layer-adapted meshes in the horizontal and vertical directions, where $x_i$ is defined in \eqref{mesh:1} and $y_j$ is defined similarly. 

Define $I_i=(x_{i-1},x_i)$ and $J_j=(y_{j-1},y_j)$ for $i,j=1,\ldots,N$. We set $h_i=x_i-x_{i-1}$ and $h_j=y_j-y_{j-1}$. Besides, let 
$$
\Omega_N=\{\text{$K_{ij}$: $K_{ij} = I_i\times J_j$ for $i,j=1,\ldots,N$}\}
$$ 
be a rectangle partition of $\Omega$. Then the finite element space is defined as: $$R_{N}=\{v\in L^2(\Omega):\text{$v|_{K_{ij}}\in \mathbb{Q}_k(K_{ij})$ }, K_{ij}\in \Omega_N \},$$ where $\mathbb{Q}_k(K_{ij})$ represents the space of polynomials on $K_{ij}$ with a maximum degree of $k$ in each variable. For $v \in R_N$ and $y \in J_j$, $j = 1, 2, \ldots ,N$, we use $v^{\pm}_{i,y}=\lim_{x\to x^{\pm}_i}v(x,y)$ and $v^{\pm}_{x,j}=\lim_{y\to y^{\pm}_j}v(x,y)$ to express the traces evaluated from the four directions.
Define the jumps on the vertical and horizontal edges:
\begin{align*}
&&&[[v]]_{i,y}=v^{-}_{i,y}-v^{+}_{i,y}\text{ for $i=1,2,\ldots,N-1$},\; &&[[v]]_{0,y}=-v^{+}_{0,y},\;&[[v]]_{N,y}=v^{-}_{N,y}.\\ 
&&&[[v]]_{x,j}=v^{-}_{x,j}-v^{+}_{x,j}\text{ for $j=1,2,\ldots,N-1$},\; &&[[v]]_{x,0}=-v^{+}_{x,0},\;&[[v]]_{x,N}=v^{-}_{x,N}.
\end{align*}

Next, we will introduce the LDG method   for problem \eqref{eq3.1}. First, one has
\begin{equation*}
\begin{aligned}
&\varepsilon^{-1}p=u_x,\quad \text{in $\Omega$}, \\
&\varepsilon^{-1}q=u_y,\quad \text{in $\Omega$}, \\
&-p_x-q_y+bu=f,\quad \text{in $\Omega$}.
\end{aligned}
\end{equation*}
Define $Z=(v,s,r)\in R_N \times R_N\times R_N$ to be any test function. Let $(\cdot,\cdot)_{D}$ be the $L^2$ inner product in the domain $D\subset \mathbb{R}^2$. Let $\left< \cdot,\cdot\right>_{I}$ be the inner product in $L^2(I)$ with $I\subset \mathbb{R}$.
Find the numerical solution $T=(U,P,Q)\in R_N \times R_N\times R_N$ such that the following variational forms hold in each element $K_{ij}$,
\begin{equation}\label{LDG2:1}
\begin{aligned}
&\varepsilon^{-1}(P,s)_{K_{ij}}+(U,s_x)_{K_{ij}}-\left< \hat{U}_{i,y}, s^{-}_{i,y}\right>_{J_j}+\left< \hat{U}_{i-1,y}, s^{+}_{i-1,y}\right>_{J_j}=0,\\
&\varepsilon^{-1}(Q,r)_{K_{ij}}+(U,r_y)_{K_{ij}}-\left<\hat{U}_{x,j}, r^{-}_{x,j}\right>_{I_i}+\left<\hat{U}_{x,j-1}, r^{+}_{x,j-1}\right>_{I_i}=0,\\
&(P,v_x)_{K_{ij}}-\left<\hat{P}_{i,y}, v^{-}_{i,y}\right>_{J_j}+\left<\hat{P}_{i-1,y}, v^{+}_{i-1,y}\right>_{J_j}+(Q,v_y)_{K_{ij}}-\left<\hat{Q}_{x,j}, v^{-}_{x,j}\right>_{I_i}\\
&+\left<\hat{Q}_{x,j-1}, v^{+}_{x,j-1}\right>_{I_i}+(bU,v)_{K_{ij}}=(f,v)_{K_{ij}}.
\end{aligned}
\end{equation}
For $y\in J_j$ and $j=1,2,\ldots,N$, we define the numerical fluxes by
\begin{equation*}
\begin{aligned}
&\hat{P}_{i,y}=
\left\{
\begin{aligned}
& P^{+}_{0,y}+\lambda_{0,y}U^{+}_{0,y}, \quad &&\text{$i=0$},\\
& P^{+}_{i,y}, \quad &&\text{$i=1,2,\ldots,N-1$},\\
& P^{-}_{N,y}-\lambda_{N,y}U^{-}_{N,y}, \quad &&\text{$i=N$},
\end{aligned}
\right. \label{eq:Bakhvalov mesh-Roos}\\
&\hat{U}_{i,y}=
\left\{
\begin{aligned}
& 0, \quad &&\text{$i=0,N$},\\
&U_{{\frac{3}{4}N},y}^{-}+\lambda_P[[P]]_{{\frac{3}{4}N},y},\quad &&\text{$i=\frac{3}{4}N$},\\
& U^{-}_{i,y}, \quad &&\text{$i=1,2,\ldots,\frac{3}{4}N-1,\frac{3}{4}N+1,\ldots,N-1$},\\
\end{aligned}
\right.\\
\end{aligned}
\end{equation*}
where the stabilization parameters $\lambda_{0,y}=\lambda_{N,y}=\sqrt{\varepsilon}$ and $\lambda_P=1/\sqrt{\varepsilon}$. Analogously, when $x\in I_i$, we can define $\hat{Q}_{x,j}$ for $i=1,2,\ldots,N$ and $\hat{U}_{x,j}$ for $j=1,2,\ldots,N$, where $\hat{U}_{x,\frac{3N}{4}}=U_{x,{\frac{3}{4}N}}^{-}+\lambda_Q[[Q]]_{x,{\frac{3}{4}N}}$ with $\lambda_Q=1/\sqrt{\varepsilon}$.

We 
express the scheme \eqref{LDG2:1} by a compact form: Find  $T=(U,P,Q)\in R_N\times R_N\times R_N$ such that
\begin{equation*}
B(T;Z)=(f,v),\quad \forall Z=(v,s,r)\in R_N\times R_N\times R_N,
\end{equation*}
where
\begin{equation}\label{LDG2:2}
\begin{aligned}
B(T;Z)=&(b U,v)+\varepsilon^{-1}(P,s)+\varepsilon^{-1}(Q,r)\\
&+(U,s_x)+\sum_{j=1}^{N}\sum_{i=1}^{N-1}\left<U^{-}_{i,y},[[s]]_{i,y}\right>_{J_j}
+(U,r_y)+\sum_{i=1}^{N}\sum_{j=1}^{N-1}\left<U^{-}_{x,j},[[r]]_{x,j}\right>_{I_i}\\
&+(P,v_x)+\sum_{j=1}^{N}\left \{ \sum_{i=0}^{N-1}\left<P^{+}_{i,y},[[v]]_{i,y}\right>_{J_j}-\left<P^{-}_{N,y},[[v]]_{N,y}\right>_{J_j} \right\}\\
&+(Q,v_y)+\sum_{i=1}^{N}\left \{ \sum_{j=0}^{N-1}\left<Q^{+}_{x,j},[[v]]_{x,j}\right>_{I_i}-\left<Q^{-}_{x,N},[[v]]_{x,N}\right>_{I_i} \right\}\\
&+\sum_{j=1}^{N}\sum_{i\in \{0,N\}}\left<\lambda_{i,y} [[U]]_{i,y},[[v]]_{i,y}\right>_{J_j}+\sum_{i=1}^{N}\sum_{j\in \{0,N\}}\left<\lambda_{x,j} [[U]]_{x,j},[[v]]_{x,j}\right>_{I_i}\\
&+\sum_{j=1}^{N}\left<\lambda_P[[P]]_{3N/4,y},[[s]]_{3N/4,y}\right>_{J_j}+\sum_{i=1}^{N}\left<\lambda_Q[[Q]]_{x,3N/4},[[r]]_{x,3N/4}\right>_{I_i}.
\end{aligned}
\end{equation}

\subsection{Regularity of the solution}
\begin{lemma}\label{lem3}
(Reference \cite{h1990-d}) Assume that the solution $u$ of \eqref{eq3.1} can be decomposed as
\begin{align*}
u=S+\sum_{k=1}^4 W_k+\sum_{k=1}^4 Z_k, \quad \text{$(x,y)\in \bar{\Omega}$},
\end{align*}
where $S$ is a smooth part, $W_k$ is a boundary layer part and $Z_k$ is a corner layer part.
Moreover, for $0 \le i, j \le k + 2$, there exists a constant $C$ independent of $\varepsilon$ such
that
\begin{equation}\label{solution2:1}
|\frac{\partial^{i+j}S}{\partial x^i \partial y^j}|\le C,\quad |\frac{\partial^{i+j}W_1}{\partial x^i \partial y^j}|\le C\varepsilon^{-\frac{i}{2}}e^{-\frac{\beta x}{\sqrt{\varepsilon}}},\quad |\frac{\partial^{i+j}Z_1}{\partial x^i \partial y^j} |\le C\varepsilon^{-\frac{i+j}{2}}e^{-\frac{\beta (x+y)}{\sqrt{\varepsilon}}},
\end{equation}
and so on for the remaining terms.
\end{lemma}


%
%

\subsection{Projections}
In this section, we combine the local $L^2$ projection and Gauss-Radau projection to obtain a new interpolation for later analytical proof.
\begin{itemize}
\item
Local $L^2$ projection $\pi$. For each element $K_{ij} \in \Omega_N$ and any $z\in L^2({\Omega})$, $\pi z\in R_N$ is defined as:
\begin{equation*}
\begin{aligned}
(\pi z,v)_{K_{ij}}=( z,v)_{K_{ij}}, \quad \forall v\in \mathbb{Q}_k(K_{ij}).
\end{aligned}
\end{equation*}
\item
Gauss-Radau projection $\pi^{\pm}$. For $l$, $m\ge 1$, we define 
$$
\mathbb{Q}_{l,m}=\left\{\text{$\sum_{i=0}^{l} \sum_{j=0}^m a_{ij}x^i y^j$:  $a_{ij}\in \mathbb{R}$}\right\}.
$$  
 For each element $K_{ij} \in \Omega_N$ and any $z\in H^1(K_{ij})$, $\pi^{-}_x z,\pi^{-}_y z, \pi^{+}_x z,\pi^{+}_y z\in \mathbb{Q}_k(K_{ij})$ are defined as:
\begin{equation*}
\begin{aligned}
&\left\{
\begin{aligned}
& (\pi^{-}_x z,v)_{K_{ij}}=(z,v)_{K_{ij}}, \quad &&\text{$\forall v\in \mathbb{Q}_{k-1,k}$},\\
& \left<(\pi^{-}_x z)^{-}_{i,y},v\right>_{J_j}=\left<z^{-}_{i,y},v\right>_{J_j}, \quad &&\forall v\in \mathbb{P}_{k}(J_j),
\end{aligned}
\right. \\
&\left\{
\begin{aligned}
& (\pi^{-}_y z,v)_{K_{ij}}=(z,v)_{K_{ij}}, \quad &&\text{$\forall v\in \mathbb{Q}_{k,k-1}$} ,\\
& \left<(\pi^{-}_y z)^{-}_{x,j},v\right>_{I_i}=\left<z^{-}_{x,j},v\right>_{I_i}, \quad &&\forall v\in \mathbb{P}_{k}(I_i),
\end{aligned}
\right. \\
&\left\{
\begin{aligned}
& (\pi^{+}_x z,v)_{K_{ij}}=(z,v)_{K_{ij}}, \quad &&\text{$\forall v\in \mathbb{Q}_{k-1,k}$} ,\\
& \left<(\pi^{+}_x z)^{+}_{i,y},v\right>_{J_j}=\left<z^{+}_{i,y},v\right>_{J_j}, \quad &&\forall v\in \mathbb{P}_{k}(J_j),
\end{aligned}
\right. \\ 
&\left\{
\begin{aligned}
& (\pi^{+}_y z,v)_{K_{ij}}=(z,v)_{K_{ij}}, \quad &&\text{$\forall v\in \mathbb{Q}_{k,k-1}$} ,\\
& \left<(\pi^{+}_y z)^{+}_{x,j},v\right>_{I_i}=\left<z^{+}_{x,j},v\right>_{I_i}, \quad &&\forall v\in \mathbb{P}_{k}(I_i).
\end{aligned}
\right. \\ 
\end{aligned}
\end{equation*}

\end{itemize}
The proof of the existence and uniqueness of Gauss-Radau projection is referred to reference \cite{A1999-An}.

For the analysis of the article, we define the following interpolation:
\begin{equation*}
\begin{aligned}
&P^{-} u|_{K_{ij}}=
\left\{
\begin{aligned}
& \pi^{-}_x u,\quad &&\text{$i=1,\ldots,N/4,3N/4+1,\ldots,N-1,\quad j=N/4+1,\ldots,3N/4$},\\
& \pi^{-}_y u,\quad &&\text{$i=N/4+1,\ldots,3N/4, \quad j=1,\ldots,N/4,3N/4+1,\ldots,N-1 $},\\
& \pi u, \quad &&\text{otherwise},\\
\end{aligned}
\right. \\
&P^{+}_x  p|_{K_{ij}}=
\left\{
\begin{aligned}
& \pi p,\quad &&\text{$i=1,\quad j=1,\ldots,N$},\\
& \pi^{+}_x p, \quad &&\text{otherwise},\\
\end{aligned}
\right.  \\
&P^{+}_y  q|_{K_{ij}}=
\left\{
\begin{aligned}
& \pi q,\quad &&\text{$i=1,\ldots,N, \quad j=1$},\\
& \pi^{+}_y q, \quad &&\text{otherwise}.\\
\end{aligned}
\right. 
\end{aligned}
\end{equation*}

Define $||z||^2=\sum_{i=1}^N\sum_{j=1}^N||z||_{K_{ij}}^2$ and $||z||^2_{K_{ij}}=(z,z)_{K_{ij}}$. Next, we will give the interpolation error estimates for $u$ and $p$, and  note that the interpolation error estimates for $q$ are similar to the ones for $p$.
\begin{lemma}\label{lem4}
\citep[Lemma 4.2 and Lemma 4.4]{m2021-local} On the two-dimensional layer-adapted meshes $\{(x_i
, y_j ), i,j=1,\ldots,N\}$ with $\sigma \ge k+1$, we have 
\begin{align}
&||u-P^{-} u||\le C\left[\varepsilon^{1/4}(N^{-1}\ln N)^{k+1}+N^{-(k+1)}\right],\label{solution2:2}\\
&||q-P^{+}_x p||\le C\varepsilon^{3/4}(N^{-1}\ln N)^{k+1},\label{solution2:3}\\
&||u-P^{-} u||_{L^{\infty}(\Omega)}\le CN^{-(k+1)},\label{solution2:4}\\
&||q-P^{+}_x p||_{L^{\infty}(\Omega)}\le C\varepsilon^{1/2}N^{-(k+1)}.\label{solution2:5}
\end{align}
\end{lemma}
\begin{proof}
The proofs of \eqref{solution2:2}, \eqref{solution2:3}, \eqref{solution2:4} and \eqref{solution2:5} refer to \citep[Lemma 4.2 and Lemma 4.4]{m2021-local}.
\end{proof}

%
%
%

\subsection{Convergence analysis}
Recall $t=(u,p,q)$ to be the exact solution of the problem \eqref{eq3.1}.
From \eqref{LDG2:2}, we define an energy norm:
\begin{align}\label{norm2:1}
|||t|||^2_{E}=B(t;t)&= \varepsilon^{-1}||p||^2+\varepsilon^{-1}||q||^2+||b(x)^{\frac{1}{2}}u||^2\\ \nonumber
             &+\sum^{N}_{j=1}\sum_{i\in\{0,N\}}(\lambda_{i,y},[[u]]_{i,y}^2)_{J_j}
+\sum^{N}_{i=1}\sum_{j\in\{0,N\}}(\lambda_{x,j},[[u]]_{x,j}^2)_{I_i}\\ \nonumber
             &+\sum_{j=1}^N(\lambda_P,[[p]]^2_{{\frac{3N}{4}},y})_{J_j}+\sum_{i=1}^N(\lambda_Q,[[q]]^2_{x,{\frac{3N}{4}}})_{I_i}.
\end{align}
And the balanced norm is defined as:
\begin{align}\label{norm2:2}
|||t|||^2_{B}&= \varepsilon^{-\frac{3}{2}}||p||^2+\varepsilon^{-\frac{3}{2}}||q||^2+||b(x)^{\frac{1}{2}}u||^2+\sum^{N}_{j=1}\sum_{i\in\{0,N\}}(1,[[u]]_{i,y}^2)_{J_j}\\ \nonumber
&+\sum^{N}_{i=1}\sum_{j\in\{0,N\}}(1,[[u]]_{x,j}^2)_{I_i}+\sum_{j=1}^N(\lambda_P,[[p]]^2_{{\frac{3N}{4}},y})_{J_j}+\sum_{i=1}^N(\lambda_Q,[[q]]^2_{x,{\frac{3N}{4}}})_{I_i}.
\end{align}

\begin{theorem}\label{theorem2}
Assume $\varepsilon \le CN^{-1}$.
On the two-dimensional layer-adapted meshes $\{(x_i, y_j ), i,j=1,\ldots,N\}$, we take $\sigma \ge k+1$. Let $T=(U,P,Q)\in R_N \times R_N\times R_N$ be the numerical solution of the LDG scheme \eqref{LDG2:1}. Let $t=(u,p,q)$ be the exact solution of the problem \eqref{eq3.1} satisfying Lemma \ref{lem3}. Define ${e=t-T}$. Besides, there is a constant $C>0$ such that
\begin{align*}
&|||e|||_{E}\le  C\left(\varepsilon^{1/4}(N^{-1}\ln N)^{k+1}+\varepsilon^{1/2}N^{-k}+N^{-(k+1)}\right),\\
&|||e|||_{B}\le C\left(\varepsilon^{1/2}(N^{-1}\ln N)^{k+1}+\varepsilon^{1/4}N^{-k}+(N^{-1}\ln N)^{k+1}\right).
\end{align*}
\end{theorem}

\begin{proof}
We use interpolation $P^{-} u$, $P^{+}_x p$ and $P^{+}_y q$ to divide the error into two parts: ${e=t-T}={\eta-\xi}$ with
\begin{equation*}
\begin{aligned}
&\bold{\eta}=(\eta_u,\eta_p,\eta_q)=(u-P^{-} u,p-P^{+}_x p, q-P^{+}_y q),\\
&\bold{\xi}=(\xi_u,\xi_p,\xi_q)=(U-P^{-} u,P-P^{+}_x p, Q-P^{+}_y q)\in R_N \times R_N\times R_N.
\end{aligned}
\end{equation*}

Owing to the consistency of numerical fluxes and the smooth assumptions of $u$, $p$ and $q$, we can obtain the error equation
\begin{equation}\label{norm2:3}
B(\xi;Z)=B(\eta;Z),\quad \forall Z=(u,p,q)\in R_N \times R_N\times R_N.
\end{equation}
Set $Z=\xi$ in \eqref{norm2:3}.  According to the integration by parts, we derive
\begin{equation}\label{norm2:4}
B(\xi;\xi)=|||\xi|||_E^2.
\end{equation}
Furthermore, by \eqref{LDG2:2}, we can get
\begin{equation*}
B(\eta;{Z})\equiv \sum_{i=1}^{15} S_i,
\end{equation*}
where
\begin{equation*}\label{eq:VN}
\begin{aligned}
&S_1=(b\eta_u,\xi_u),\quad S_2=\varepsilon^{-1}(\eta_p,\xi_p),\\
&S_3=(\eta_u,\xi _{p,x}),\quad S_4=(\eta_p,\xi_{u,x}),\\
&S_5=\sum_{j=1}^{N}\sum_{i=1}^{N-1}\left<(\eta_u)^{-}_{i,y},[[\xi_p]]_{i,y}\right>_{J_j},\\
&S_6=\sum_{j=1}^{N}\left \{ \sum_{i=0}^{N-1}\left<(\eta_p)^{+}_{i,y},[[\xi_u]]_{i,y}\right>_{J_j}-\left<(\eta_p)^{-}_{N,y},[[\xi_u]]_{N,y}\right>_{J_j} \right\},\\
&S_7=\sum_{j=1}^{N}\sum_{i\in \{0,N\}}\left<\lambda_{i,y} [[\eta_u]]_{i,y},[[\xi_u]]_{i,y}\right>_{J_j},\\
&S_8=\sum_{j=1}^{N}\left<\lambda_P[[\eta_p]]_{3N/4,y},[[\xi_p]]_{3N/4,y}\right>_{J_j}.
\end{aligned}
\end{equation*}\label{eq:VN}
Here, we only prove the convergence analysis for the part of $p$, because the remaining proof for $q$ is similar to $p$. Furthermore, we will analyze the above parts separately.

Using the Cauchy–Schwarz inequality, \eqref{solution2:2} and \eqref{solution2:3}, one has
\begin{align*}
&S_1\le C||b^{\frac{1}{2}}\eta_u||||b^{\frac{1}{2}}\xi_u||\le C\varepsilon^{1/4}(\ln N)^{1/2} N^{-(k+1)}|||\xi|||_E.\\
&S_2\le (\varepsilon^{-1/2}||\eta_p||)(\varepsilon^{-1/2}||\xi_p||)\le C\varepsilon^{1/4}(N^{-1}\ln N)^{k+1}|||\xi|||_E.
\end{align*}
Owing to the orthogonality of the projection, we have $S_3=S_4=0$.\\
Let $\Omega_c=\Omega-\left\{\left((0,x_{N/4})\times (y_{N/4},y_{3N/4})\right)\cup \left((x_{3N/4},x_{N-1})\times (y_{N/4},y_{3N/4})\right)\right\}$. From the Cauchy–Schwarz inequality, the trace inequality, \eqref{solution2:4}, $\varepsilon \le CN^{-1}$ and $\lambda_P=\varepsilon^{-1/2}$, we can get
\begin{align*}
S_5&\le C\sum_{K_{ij}\in \Omega_c,i\neq 3N/4}||(\eta_u)^{-}_{i,y}||_{k_{ij}}||[[\xi_p]]_{i,y}||_{k_{ij}}+\sum_{j=1}^{N}\left<(\eta_u)^{-}_{3N/4,y},[[\xi_p]]_{3N/4,y}\right>_{J_j}\\
&\le Ch_j^{1/2}h_j^{1/2}(h_i h_j)^{-1/2}||\eta_u||_{L^{\infty}} N ||\xi_p||+C(\sum_{j=1}^{N}\lambda_p^{-1} h_j||(\eta_u)||_{L^{\infty}(k_{ij})}^2)^{1/2}|||\xi|||_E\\
&\le C\varepsilon^{1/2}N^{-k}|||\xi|||_E+C\varepsilon^{1/4} N^{-(k+1)}|||\xi|||_E.
\end{align*}
By the Cauchy–Schwarz inequality, \eqref{solution2:5}, \eqref{solution2:4} and $\lambda_{0,y}=\lambda_{N,y}=\sqrt{\varepsilon}$, one has
\begin{align*}
&S_6\le C\varepsilon^{1/4}N^{-(k+1)}|||\xi|||_E.\\
&S_7\le C\varepsilon^{1/4}N^{-(k+1)}|||\xi|||_E.
\end{align*}
%
From the Cauchy–Schwarz inequality, $\lambda_P =\varepsilon^{-1/2}$ and \eqref{solution2:5}, we have
\begin{equation*}
\begin{aligned}
S_8
   \le C\varepsilon^{1/4}N^{-(k+1)}|||\xi|||_E.
\end{aligned}
\end{equation*}
We gather all the above results and get
\begin{equation*}
\begin{aligned}
B(\eta;\xi)\le C\varepsilon^{1/4}(N^{-1}\ln N)^{k+1}|||\xi|||_{E}.
\end{aligned}
\end{equation*}
Considering \eqref{norm2:3} and \eqref{norm2:4}, we have
\begin{equation*}
|||\xi|||^2_{E}= B(\xi;\xi)=B(\eta;\xi)\le C\varepsilon^{1/4}(N^{-1}\ln N)^{k+1}|||\xi|||_{E}.
\end{equation*}
Therefore, it is straightforward to derive
\begin{align}
&|||\xi|||_{E}\le C\varepsilon^{1/4}(N^{-1}\ln N)^{k+1}+C\varepsilon^{1/2}N^{-k},\label{norm2:5}\\
&|||\xi|||_{B}\le C\varepsilon^{-1/4}|||\xi|||_E\le C(N^{-1}\ln N)^{k+1}+C\varepsilon^{1/4}N^{-k}.\label{norm2:6}
\end{align} 
From Lemma \ref{lem4} and the definition of norms in \ref{norm2:1} and \ref {norm2:2}, we can get
\begin{align}
&|||\eta|||_{E}\le C\varepsilon^{1/4}(N^{-1}\ln N)^{k+1}+CN^{-(k+1)},\label{norm2:7}\\
&|||\eta|||_{B}\le C\varepsilon^{1/2}(N^{-1}\ln N)^{k+1}+C(N^{-1}\ln N)^{k+1}.\label{norm2:8}
\end{align} 
Combining the results of \eqref{norm2:5}, \eqref{norm2:6}, \eqref{norm2:7} and \eqref{norm2:8}, we prove Theorem \ref{theorem2}.
\end{proof}
\begin{remark}
In Theorem \ref{theorem2}, if  $\varepsilon^{} \le CN^{-4}$, then we have 
$$|||e|||_{E}\le C(N^{-1}\ln N)^{k+1},\quad
|||e|||_{B}\le C(N^{-1}\ln N)^{k+1}.$$
Similarly, we obtain optimal convergence order without the impossible condition in \cite{c2022-balanced}.
\end{remark}

\section{Numerical experiments}
In this section, we mainly do some numerical examples to verify whether Theorem \ref{theorem1} is sharp or not.
Consider the following one-dimensional problem
\begin{equation*}
-\varepsilon u''+ u=f(x)\quad \text{in $\Omega:=(0,1)$},\quad
u(0)=u(1)=0.
\end{equation*}
We choose the right-hand side $f$ such that the exact solution of the problem is 
\begin{equation*}
u=\frac{1-e^{-1/\sqrt{\varepsilon}}}{e^{-x/\sqrt{\varepsilon}}-e^{-(1-x)/\sqrt{\varepsilon}}}-\cos(\pi x).
\end{equation*}

We set $\varepsilon=10^{-4},10^{-6},\ldots,10^{-12}$; $N=32,64,\ldots,1024$ and $k=1,2,3$. On the Shishkin mesh \eqref{mesh:1}, we take $\sigma=k+1$.
 Let $e^N$ be the error of $|||e|||_B$, where
$N$ is the element number under the computation. We calculate the convergence rates
with the following formulas
\begin{equation*}
r_s=\frac{\log e^N-\log e^{2N}}{\log 2},\quad r_p=\frac{\log e^N-\log e^{2N}}{\log (2\ln N/\ln 2N)}.
\end{equation*}

Table \ref{table-W-1} shows energy norm errors and convergence rates. Balanced norm errors and convergence rates are displayed in Table \ref{table-W-2}.
The data in Table \ref{table-W-1} and Table \ref{table-W-2} imply that Theorem \ref{theorem1} is sharp. Additionally, we can conclude that the result is approximately $|||e|||_B=\varepsilon^{-1/4}|||e|||_E$ by comparing $|||e|||_E$ and $|||e|||_B$ under each $\varepsilon$.
Through observation, we find that the error result of $|||e|||_E$ and $|||e|||_B$ are mainly controlled by $\varepsilon^{1/4}(N^{-1}\ln N)^{k+1}$ and $(N^{-1}\ln N)^{k+1}$, respectively. 
The numerical experiments in 2D are not presented since they are analogous to those in 1D.
\begin{table}[H] 
\caption{balanced norm errors and convergence rates $r_s$}\label{table-W-1}
\footnotesize
\begin{tabular*}{\textwidth}{@{\extracolsep{\fill}} cccccccccccc }
\cline{1-12}
      \multirow{2}{*}{}     &\multirow{2}{*}{ }   &\multicolumn{2}{c}{$\varepsilon=10^{-4}$} &\multicolumn{2}{c}{$\varepsilon=10^{-6}$}  &\multicolumn{2}{c}{$\varepsilon=10^{-8}$}   
&\multicolumn{2}{c}{$\varepsilon=10^{-10}$} &\multicolumn{2}{c}{$\varepsilon=10^{-12}$}   \\
 
\cline{3-12}
$k$ &$N$&$|||e|||_B$&$r_s$&$|||e|||_B$&$r_s$&$|||e|||_B$&$r_s$&$|||e|||_B$&$r_s$&$|||e|||_B$&$r_s$\\
\cline{1-12}

$1$       & $32$       & 0.25E+0  & 1.15    & 0.25E+0  & 1.15  & 0.25E+0   & 1.15  & 0.25E+0   & 1.15  &0.25E+0  &1.15\\
$ $       &$64$        & 0.11E+0  & 1.34    & 0.11E+0  & 1.34  & 0.11E+0   & 1.34  & 0.11E+0   & 1.34  &0.11E+0  &1.34\\
$ $       &$128$       & 0.44E-1  & 1.48    & 0.44E-1  & 1.48  & 0.44E-1   & 1.48  & 0.44E-1   & 1.48  &0.44E-1  &1.48\\
$ $       &$256$       & 0.16E-1  & 1.58    & 0.16E-1  & 1.58  & 0.16E-1   & 1.58  & 0.16E-1   & 1.58  &0.16E-1  &1.58\\
$ $       &$512$       & 0.53E-2  & 1.65    & 0.53E-2  & 1.65  & 0.53E-2   & 1.65  & 0.53E-2   & 1.65  &0.53E-2  &1.65\\
$ $       &$1024$      & 0.17E-2  & ---    & 0.17E-2  & ---  & 0.17E-2   & ---  & 0.17E-2   & ---  &0.17E-2  &---\\

$2$       & $32$       & 0.61E-1  & 1.78    & 0.61E-1  & 1.78  & 0.61E-1   & 1.78  & 0.61E-1   & 1.78  &0.61E-1  &1.78\\
$ $       &$64$        & 0.18E-1  & 2.05    & 0.18E-1  & 2.05  & 0.18E-1   & 2.05  & 0.18E-1   & 2.05  &0.18E-1  &2.05\\
$ $       &$128$       & 0.43E-2  & 2.25    & 0.43E-2  & 2.25  & 0.43E-2   & 2.25  & 0.43E-2   & 2.25  &0.43E-2  &2.25\\
$ $       &$256$       & 0.90E-3  & 2.39    & 0.90E-3  & 2.39  & 0.90E-3   & 2.39  & 0.90E-3   & 2.39  &0.90E-3  &2.39\\
$ $       &$512$       & 0.17E-3  & 2.48    & 0.17E-3  & 2.48  & 0.17E-3   & 2.48  & 0.17E-3   & 2.48 &0.17E-3  &2.48\\
$ $       &$1024$      & 0.31E-4  & ---    & 0.31E-4  & ---  & 0.31E-4   & ---  & 0.31E-4   & ---  &0.31E-4  &---\\

$3$       & $32$       & 0.15E-1  & 2.41    & 0.15E-1  & 2.41  & 0.15E-1   & 2.41  & 0.15E-1   & 2.41 
& 0.15E-1   & 2.41\\
$ $       &$64$        & 0.28E-2  & 2.76    & 0.28E-2  & 2.76  & 0.28E-2   & 2.76  & 0.28E-2   & 2.76
& 0.28E-2   & 2.76\\
$ $       &$128$       & 0.41E-3  & 3.01    & 0.41E-3  & 3.01  & 0.41E-3   & 3.01  & 0.41E-3   & 3.01 &0.41E-3  & 3.01\\
$ $       &$256$       & 0.51E-4  & 3.19    & 0.51E-4  & 3.19  & 0.51E-4   & 3.19  & 0.51E-4   & 3.19   &0.51E-4  & 3.19\\
$ $       &$512$       & 0.55E-5  & ---    & 0.55E-5  & ---  & 0.55E-5   & ---  & 0.55E-5   &---  &0.55E-5  &---\\

\cline{1-12}
\end{tabular*}
\label{table:2}
\end{table}
%

\begin{table}[H] 
\caption{balanced norm errors and convergence rates $r_p$}\label{table-W-2}
\footnotesize
\begin{tabular*}{\textwidth}{@{\extracolsep{\fill}} cccccccccccc }
\cline{1-12}
      \multirow{2}{*}{}     &\multirow{2}{*}{ }   &\multicolumn{2}{c}{$\varepsilon=10^{-4}$} &\multicolumn{2}{c}{$\varepsilon=10^{-6}$}  &\multicolumn{2}{c}{$\varepsilon=10^{-8}$}   
&\multicolumn{2}{c}{$\varepsilon=10^{-10}$} &\multicolumn{2}{c}{$\varepsilon=10^{-12}$}   \\
 
\cline{3-12}
$k$ &$N$&$|||e|||_B$&$r_p$&$|||e|||_B$&$r_p$&$|||e|||_B$&$r_p$&$|||e|||_B$&$r_p$&$|||e|||_B$&$r_p$\\
\cline{1-12}

$1$       & $32$       & 0.25E+0  & 1.56    & 0.25E+0  & 1.56  & 0.25E+0   & 1.56  & 0.25E+0   & 1.56  &0.25E+0  &1.56\\
$ $       &$64$        & 0.11E+0  & 1.72    & 0.11E+0  & 1.72  & 0.11E+0   & 1.72  & 0.11E+0   & 1.72  &0.11E+0  &1.72\\
$ $       &$128$       & 0.44E-1  & 1.83    & 0.44E-1  & 1.83  & 0.44E-1   & 1.83  & 0.44E-1   & 1.83  &0.44E-1  &1.83\\
$ $       &$256$       & 0.16E-1  & 1.90    & 0.16E-1  & 1.90  & 0.16E-1   & 1.90  & 0.16E-1   & 1.90  &0.16E-1  &1.90\\
$ $       &$512$       & 0.53E-2  & 1.95    & 0.53E-2  & 1.95  & 0.53E-2   & 1.95  & 0.53E-2   & 1.95  &0.53E-2  &1.95\\
$ $       &$1024$      & 0.17E-2  & ---    & 0.17E-2  & ---  & 0.17E-2   & ---  & 0.17E-2   & ---  &0.17E-2  &---\\

$2$       & $32$       & 0.61E-1  & 2.42    & 0.61E-1  & 2.42  & 0.61E-1   & 2.42  & 0.61E-1   & 2.42 &0.61E-1  &2.42\\
$ $       &$64$        & 0.18E-1  & 2.64    & 0.18E-1  & 2.64  & 0.18E-1   & 2.64  & 0.18E-1   & 2.64  &0.18E-1  &2.64\\
$ $       &$128$       & 0.43E-2  & 2.79    & 0.43E-2  & 2.79  & 0.43E-2   & 2.79  & 0.43E-2   & 2.79  &0.43E-2  &2.79\\
$ $       &$256$       & 0.90E-3  & 2.88    & 0.90E-3  & 2.88  & 0.90E-3   & 2.88  & 0.90E-3   & 2.88  &0.90E-3  &2.88\\
$ $       &$512$       & 0.17E-3  & 2.92    & 0.17E-3  & 2.92  & 0.17E-3   & 2.92  & 0.17E-3   & 2.92 &0.17E-3  &2.92\\
$ $       &$1024$      & 0.31E-4  & ---    & 0.31E-4  & ---  & 0.31E-4   & ---  & 0.31E-4   & ---  &0.31E-4  &---\\

$3$       & $32$       & 0.15E-1  & 3.27    & 0.15E-1  & 3.27  & 0.15E-1   & 3.27  & 0.15E-1   & 3.27 
& 0.15E-1   &3.27\\
$ $       &$64$        & 0.28E-2  & 3.55    & 0.28E-2  & 3.55  & 0.28E-2   & 3.55  & 0.28E-2   & 3.55
& 0.28E-2   & 3.55\\
$ $       &$128$       & 0.41E-3  & 3.73    & 0.41E-3  & 3.73  & 0.41E-3   & 3.73  & 0.41E-3   & 3.73 &0.41E-3  & 3.73\\
$ $       &$256$       & 0.51E-4  & 3.84    & 0.51E-4  & 3.84  & 0.51E-4   & 3.84  & 0.51E-4   & 3.84   &0.51E-4  & 3.84\\
$ $       &$512$       & 0.55E-5  & ---    & 0.55E-5  & ---  & 0.55E-5   & ---  & 0.55E-5   &---  &0.55E-5  &---\\

\cline{1-12}
\end{tabular*}
\label{table:2}
\end{table}

\bibliographystyle{plain}


\end{document}